\documentclass[twoside, 12pt]{article}
\usepackage[russian, english]{babel}
\usepackage{epsfig}
\usepackage{amssymb,amsmath,amsfonts,soul,amsthm,enumerate}
\usepackage{amsfonts}
\usepackage{amsmath}
\usepackage{graphicx}
\usepackage{amsthm}
\usepackage[cp1251]{inputenc}
\usepackage[T2A]{fontenc}
\usepackage{mathrsfs}
\usepackage[english]{babel}
\newtheorem{theorem}{Theorem}[section]
\newtheorem{lemma}{Lemma}[section]
\newtheorem{definition}{Definition}

\newtheorem{corollary}{Corollary}[section]
\numberwithin{equation}{section}
 \def\@evenhead{\vbox{\hbox to \textwidth{\thepage\hfil\sl\leftmark\strut}\hrule}}
 \def\@oddhead{\vbox{\hbox to \textwidth{\rightmark\hfill\thepage\strut}\hrule}}

\begin{document}
 \sloppy

\begin{center}
\textbf {ON THE BOUNDEDNESS OF THE MAXIMAL AND FRACTIONAL MAXIMAL, POTENTIAL OPERATORS IN THE GLOBAL MORREY-TYPE SPACES WITH VARIABLE EXPONENTS}
\end{center}

\vskip 0.3cm

\centerline{\textbf{N.A.~Bokayev, Zh.M.~Onerbek}}
\markboth{\hfill{\footnotesize\rm  N.\,A.~Bokayev, Zh.\,M.~Onerbek}\hfill}
{\hfill{\footnotesize\sl On boundedness of the Riesz potential and fractional maximal operators}\hfill}
\vskip 0.3cm

\vskip 0.7 cm


\vskip 0.2cm

\vskip 0.2cm

\noindent\textbf{Abstract.} We consider the global Morrey-type spaces ${GM}_{p(.),\theta(.),w(.)}(\Omega)$ with variable exponents $p(x)$, $\theta(x)$ and general function $w(x,r)$ defining these spaces. In the case of unbounded sets $\Omega\subset{\mathbb{R}}^{n}$, we prove boundedness of the Hardy--Littlewood maximal operator, potential type operator in these spaces.
\section{Introduction}
In this paper we consider the global Morrey-type spaces ${GM}_{p(.),\theta(.),w(.)}(\Omega)$ with variable exponents $p(.)$, $\theta(.)$ and a general function $w(x,r)$ defining a Morrey-type norm.
The Morrey spaces ${M}_{p,\lambda}$ are introduced in~{\cite{Morrey}} in relation to the study of partial differential equations. Many classical operators of harmonic analysis (for example, maximal, fractional maximal, potential operators) were studied in the  Morrey-type spaces with constant exponents $p$, $\theta$~{\cite{Burenkov1, Burenkov2, Burenkov3}}.
The Morrey spaces also attracted attention of researchers in the area of variable exponent analysis; see~{\cite{Almeida1, Almeida2, Almeida3, Almeida4, Almeida5, Alvarez}}.
The  Morrey spaces ${\mathcal{L}}_{p(.),\lambda(.)}$ with variable exponent $p(.)$, $\lambda(.)$ were introduced and studied in~{\cite{Almeida1}}. The general version ${M}_{p(.),w(.)}(\Omega)$, $\Omega\subset{\mathbb{R}}^{n}$ were introduced and studied in~{\cite{Guliyev2}} in the case of bounded sets $\Omega\subset{\mathbb{R}}^{n}$, and in~{\cite{Guliyev}} in the case of unbounded sets $\Omega\subset{\mathbb{R}}^{n}$.
The boundedness of maximal and potential type operators in the generalized Morrey-type spaces with a variable exponent were considered in~{\cite{Guliyev2}} in the case of bounded sets $\Omega\subset{\mathbb{R}}^{n}$, in~{\cite{Guliyev}} in the case of unbounded sets $\Omega\subset{\mathbb{R}}^{n}$.

Let $f\in{L}_{loc}({\mathbb{R}}^{n})$.The  Hardy--Littlewood maximal operator is defined as
\begin{equation*}
Mf(x)=\mathop{sup}_{r>0}\frac{1}{|B(x,r)|}\int_{\tilde{B}(x,r)}|f(y)|dy
\end{equation*}
where $B(x,r)$ is the ball in ${\mathbb R}^{n}$ centered at the point $x\in{\mathbb{R}}^{n}$ and of the radius $r$, $\tilde{B}(x,r)=B(x,r)\cap{\Omega}$, $\Omega\subset{\mathbb{R}}^{n}$.

The fractional maximal operator of variable order $\alpha(x)$ is defined as
\begin{equation*}
{M}^{\alpha(.)}f(x)=\mathop{sup}_{r>0}{|B(x,r)|}^{-1+\frac{\alpha(x)}{n}}\int_{\tilde{B}(x,r)}|f(y)|dy, 0\leq\alpha(x)<n.
\end{equation*}
In the case of $\alpha(x)=\alpha=const$, this operator coincides with the classical fractional maximal operator ${M}^{\alpha}$. If $\alpha(x)=0$, then ${M}^{\alpha(.)}$ coincides with the operator $M$.

The Riesz potential ${I}^{\alpha(x)}$ of variable order  $\alpha(x)$ is defined by the following equality:
$$
{I}^{\alpha(x)}f(x)=\int_{{\mathbb{R}}^{n}}\frac{f(y)}{\mathop{\left|x-y\right|}
\nolimits^{n-\alpha(x)}}dy,   0<\alpha(x)<n.
$$
In the case of $\alpha(x)=\alpha=const$, this operator coincides with the classical Riesz potential ${I}^{\alpha}$.

\section{Preliminaries.Variable Exponent Lebesgue Spaces~${L}_{p(.)}$. Generalized variable exponent Morrey spaces ${M}_{p(.),w(.)}$}
Let $p(x)$ be a measurable function on an open set~$\Omega\subset{\mathbb{R}}^{n}$ with values $(1,\infty)$.
Let
\begin{equation}\label{2.1}
1<p_{-}\leq p(x)\leq p_{+}<\infty
\end{equation}
where
${p}_{-}={p}_{-}(\Omega)=\mathop{essinf}_{x\in\Omega}p(x)$, ${p}_{+} ={p}_{+}(\Omega)=\mathop{essup}_{x\in\Omega}p\left(x\right)$.
We denote by  $\mathop{L}\nolimits_{p(.)}(\Omega)$ the space of all measurable functions $f(x)$ on $\Omega$ such that
$$
\mathop{J}\nolimits_{p(.)} (f)=\int_{\Omega}\mathop{\left[f(x)\right]}\nolimits^{p(x)} dx<\infty,
$$
where the norm is defined as follows
$$
\mathop{\left|\left|f\right|\right|}\nolimits_{p\left(.\right)} =\inf \left\{\eta >0,\mathop{J}\nolimits_{p(.)} \right. \left(\frac{f}{\eta} \right)\le 1\left.\right\}.
$$
This is a Banach space. The conjugate exponent ${p}^{'}$ is defined by the formula
$$
{p}^{'}(x)=\frac{p(x)}{p(x)-1}
$$
Holder inequality for the variable exponents $p(.)$, ${p}^{'}(.)$ is
$$
\int_{\Omega}f(x)g(x)dx\leq{C(p)}{\|f\|}_{{L}_{p(.)}(\Omega)}{\|g\|}_{{L}_{{p}^{'}(.)}(\Omega)},
$$
where $C(p)=\frac{1}{{p}_{-}}+\frac{1}{{p}_{-}^{'}}$.

The variable exponent Lebesgue spaces ${L}_{p(.)}$ were introduced in~{\cite{Sharafutdtinov}}, and were investigated in~{\cite{Diening1, Diening2}}.

$\mathcal{P}(\Omega)$ is the set of measurable functions $p:\Omega\rightarrow[1,\infty)$,
${\mathcal{P}}^{\log}(\Omega)$  is the set of  measurable functions $p(x)$ satisfying the
local log-condition
$$
\left|p(x)-p(y)\right|\leq\frac{{A}_{p}}{-ln\left|x-y\right|},
:\left| x-y\right|\leq\frac{1}{2}, {x},y\in\Omega
$$
where ${A}_{p}$ is independent of $x$ и $y$.
${\mathbb{P}}^{\log}(\Omega)$ is the set of  measurable functions $p(x)$ satisfying~{\eqref{2.1}} and the log-condition.
In the case of $\Omega$ is an unbounded set, we denote by ${\mathbb{P}}_{\infty}^{log}(\Omega)$ the set of exponents which is a subset of the set of ${\mathbb{P}}^{\log}(\Omega)$ and satisfying the decay condition
\begin{equation*}
|p(x)-p(\infty)|\leq{A}_{\infty}ln(2+|x|),  x\in{\mathbb{R}}^{n}.
\end{equation*}
Let ${\mathbb{A}}^{log}(\Omega)$ be the set of bounded exponents $\alpha:\Omega\rightarrow{\mathbb{R}}$ satisfying the log-condition.

Let $\Omega$ be an open bounded set, $p\in{\mathbb{P}}^{\log}(\Omega)$ and $\lambda(x)$ be a measurable function on $\Omega$ with values in $[0,n]$. The variable Morrey space ${\mathcal{L}}_{p(.),\lambda(.)}(\Omega)$ is introduced in~{\cite{Burenkov2}} with the norm
$$
{\|f\|}_{{\mathcal{L}}_{p(.),\lambda(.)}(\Omega)}=\mathop{sup}_{x\in\Omega,t>0}{t}^{-\frac{\lambda(x)}{p(x)}}{\|f\|}_{{L}_{p(.)}(\tilde{B}(x,t))}.
$$

Let $w(x,r)$ be nonnegative measurable function on $\Omega$, where $\Omega \subset {\mathbb{R}}^{n}$ is a open bounded set. The generalized Morrey type space ${M}_{p(.),w(.)}(\Omega)$ with variable exponent is defined in~{\cite{Guliyev2}} with the norm
$$
{||f||}_{{M}_{p(.),w(.)}(\Omega)}=\mathop{sup}_{x\in\Omega,r>0}\frac{{r}^{-\frac{n}{p(x)}}}{w(x,r)}{||f||}_{{L}_{p(.)}(\tilde{B}(x,r))}.
$$
Let $w(x,r)$ be nonnegative measurable function on $\Omega$, where $\Omega \subset {\mathbb{R}}^{n}$ is a open unbounded set .The generalized Morrey type space ${M}_{p(.),w(.)}(\Omega)$ with variable exponent is defined in~{\cite{Guliyev}} with the norm
$$
{||f||}_{{M}_{p(.),w(.)}(\Omega)}=\mathop{sup}_{x\in\Omega,r>0}\frac{{||f||}_{{L}_{p(.)}(\tilde{B}(x,r))}}{w(x,r)}.
$$
Let
$$
{\eta}_{p}(x,r)=\begin{cases}
\frac{n}{p(x)},&\text{if $r\leq1$;}\\
\frac{n}{p(\infty)},&\text{if $r>1$.}
\end{cases}
$$
\begin{definition}
{\rm
Let $p\in{P}^{log}(\Omega)$, $w(x,r)$ be a positive function on $\Omega\times[0,\infty]$, where $\Omega\in{\mathbb R}^{n}$.
Global Morrey-type space with variable exponent ${GM}_{p(.),\theta(.),w(.)}(\Omega)$ is defined as the set of functions $f\in{L}_{p(.)}^{loc}(\Omega)$ with finite norm
\begin{equation*}
{||f||}_{{GM}_{p(.),\theta(.),w(.)}(\Omega)}=\mathop{sup}_{x\in\Omega}{||w(x,r){r}^{-{\eta}_{p}(x,r)}{||f||}_{{L}_{p(.)}(\tilde{B}(x,r))}||}_{{L}_{\theta(.)}(0,\infty)}.
\end{equation*}
}
\end{definition}

We assume that the positive measurable function $w(x,r)$ satisfies the condition
$$
\mathop{sup}_{x\in\Omega}{||w(x,r)||}_{{L}_{\theta(.)}(0,\infty)}<\infty.
$$
Then the space contains bounded functions and thereby is nonempty.
In the case of $w(x,r)={r}^{-\frac{\lambda(x)}{p(x)}+{\eta}_{p}(x,r)}$, the corresponding space is denoted by ${GM}_{p(.),\theta(.)}^{\lambda(.)}$:
$$
{GM}_{p(.),\theta(.)}^{\lambda(.)}(\Omega)={{GM}_{p(.),w(.),\theta}|}_{w(x,r)={r}^{-\frac{\lambda(x)}{p(x)}+{\eta}_{p}(x,r)}},
$$
$$
{||f||}_{{GM}_{p(.),\theta(.)}^{\lambda(.)}(\Omega)}=
\mathop{sup}_{x\in\Omega}{||w(x,r){r}^{-\frac{\lambda(x)}{p(x)}}{||f||}_{{L}_{p(.)}(\tilde{B}(x,r))}||}_{{L}_{\theta(.)}(0,\infty)}.
$$

In the case of $\theta=\infty$, the  space ${GM}_{p(.),\infty,w(.)}(\Omega)$ coincides the generalized Morrey space with variable exponent ${M}_{p(.),w(.)}(\Omega)$ with finite quasi-norm
$$
{||f||}_{{M}_{p(.),w(.)}(\Omega)}=\mathop{sup}_{x\in\Omega}w(x,r){r}^{-{\eta}_{p}(x,r)}{||f||}_{{L}_{p(.)}(\tilde{B}(x,r))}.
$$
If $p(.)=p=const$, $\theta(x)=\theta=const$ then the space ${GM}_{p(.),\theta(.),w(.)}(\Omega)$ coincides with the ordinary global Morrey space ${GM}_{p,\theta,w}(\Omega)$, considered in the works by V.I.~Burenkov, V.~Guliev and others~{\cite{Burenkov1, Burenkov2, Burenkov3}}.

The Spanne and The Adams type theorems were proved in~{\cite{Almeida1}} for bounded sets $\Omega$.

\begin{theorem}Suppose that $p\in{\mathbb{P}}_{\infty}^{log}(\Omega)$ and
\begin{equation*}
\mathop{sup}_{t>r}\frac{\mathop{essinf}_{t<s<\infty}{w}_{1}(x,s)}{{t}^{{\eta}_{p}(x,t)}}\leq{C}\frac{{w}_{2}(x,r)}{{r}^{{\eta}_{p}(x,r)}}
\end{equation*}
where $C$ is independent of $x$ and $r$. Then the maximal operator $M$ from ${M}_{p(.),{w}_{1}(.)}(\Omega)$ to ${M}_{p(.),{w}_{2}(.)}(\Omega)$  is bounded.
\end{theorem}
\begin{theorem}[Spanne type result with $\alpha=const$]\label{th2.2}
Let $p\in{\mathbb{P}}_{\infty}^{log}(\Omega)$, $\alpha$ and $p(.),q(.)$ satisfy $0<\alpha<n$, $\frac{1}{q(x)}=\frac{1}{p(x)}-\frac{\alpha}{n}$ , the functions ${w}_{1}$ and ${w}_{2}$ satisfy the condition
\begin{equation*}
\int_{r}^{\infty}\frac{\mathop{essinf}_{t\leq{s}<\infty}{w}_{1}(x,s)}{{t}^{1+{\eta}_{p}(x,t)}}dt\leq{C}\frac{{w}_{2}(x,r)}{{r}^{{\eta}_{q}(x,r)}}
\end{equation*}
where $C$ is independent of $x$ and $r$. Then the operators ${M}_{\alpha}$ and ${I}_{\alpha}$ from ${M}_{p(.),{w}_{1}(.)}$ to ${M}_{q(.),{w}_{2}(.)}$ are bounded.
\end{theorem}

The next theorem was proved in~{\cite{Kokilashvili}}.
\begin{theorem}\label{th2.3} Suppose that $p\in{\mathbb{P}}_{\infty}^{log}(\Omega)$, $\alpha\in{\mathbb{A}}^{log}(\Omega)$ and ${\alpha}_{-}=\mathop{inf}_{x\in\Omega}\alpha(x)>0$, ${(\alpha{p})}_{+}=\mathop{sup}_{x\in\Omega}\alpha(x)p(x)<n$. Then
\begin{equation*}
{\|\frac{1}{{(1+|x|)}^{\gamma(x)}}{I}^{\alpha(.)}f\|}_{{L}_{q(.)}({R}^{n})}\leq{C}{\|f\|}_{{L}_{p(.)}({R}^{n})},
\end{equation*}
where
$$
\frac{1}{q(x)}=\frac{1}{p(x)}-\frac{\alpha(x)}{n},
\gamma(x)={A}_{\infty}\alpha(x)[1-\frac{\alpha(x)}{n}]\leq\frac{n}{4}{A}_{\infty}
$$
with ${A}_{\infty}$ comes from~{\eqref{2.1}}.
\end{theorem}

The following results were obtained in~{\cite{Guliyev}}.

\begin{theorem} Suppose that $p\in{\mathcal{P}}^{\log}(\Omega)$ satisfies~{\eqref{2.1}}, and let
$1<{\theta}_{1}^{-}\leq{\theta}_{1}(t)\leq{\theta}_{1}^{+}<\infty$, $0<t<l$,
$1<{\theta}_{2}^{-}\leq{\theta}_{2}(t)\leq{\theta}_{2}^{+}<\infty$, $0<t<l$.
Assume that there exists $\delta>0$ such that ${\theta}_{1}(t)\leq{\theta}_{2}(t)$, $t\in(0,\delta)$,
$({\theta}_{1},{w}_{1})\in\mathcal{W}(\delta,l)$.
If
\begin{equation*}
\mathop{sup}_{x\in\Omega,0<t<\delta}\int_{0}^{t}{({w}_{2}(x,\xi))}^{{\theta}_{2}(\xi)}{(\int_{t}^{\delta}{(\frac{1}{r{w}_{1}(x,r)})}^{{[\tilde{\theta}_{1}(\xi)]}^{'}}dr)}^{\frac{{\theta}_{2}(\xi)}{{[\tilde{\theta}_{1}(\xi)]}^{'}}}d\xi<\infty,
\end{equation*}
then the maximal operator $M$ from ${M}_{p(.),{\theta}_{1}(.),{w}_{1}(.)}(\Omega)$ to ${M}_{p(.),{\theta}_{2}(.),{w}_{2}(.)}(\Omega)$ is bounded.
\end{theorem}
\begin{theorem} Suppose that $p,\alpha\in{\mathcal{P}}^{\log}(\Omega)$ satisfy~{\eqref{2.1}}, and  let $\alpha>0$, ${(\alpha p(.))}_{+}=\mathop{sup}_{x\in\Omega}\alpha p(x)<n$, $\frac{1}{{p}_{2}(x)}=\frac{1}{{p}_{1}(x)}-\frac{\alpha}{n}$
$1<{\theta}_{1}^{-}\leq{\theta}_{1}(t)\leq{\theta}_{1}^{+}<\infty$, $0<t<l$,
$1<{\theta}_{2}^{-}\leq{\theta}_{2}(t)\leq{\theta}_{2}^{+}<\infty$, $0<t<l$.
Assume that there exists $\delta>0$ such that ${\theta}_{1}(t)\leq{\theta}_{2}(t)$, $t\in(0,\delta)$,
$({\theta}_{1},{w}_{1})\in\mathcal{W}(\delta,l)$.
If 
\begin{equation*}
\mathop{sup}_{x\in\Omega,0<t<\delta}\int_{0}^{t}{({w}_{2}(x,\xi))}^{{\theta}_{2}(\xi)}{(\int_{t}^{\delta}{(\frac{{r}^{\alpha(x)-1}}{{w}_{1}(x,r)})}^{{[\tilde{\theta}_{1}(\xi)]}^{'}}dr)}^{\frac{{\theta}_{2}(\xi)}{{[\tilde{\theta}_{1}(\xi)]}^{'}}}d\xi<\infty,
\end{equation*}
then the operators ${I}_{\alpha}$ и ${M}_{\alpha}$ from ${M}_{{p}_{1}(.),{\theta}_{1}(.),{w}_{1}(.)}(\Omega)$ to ${M}_{{p}_{2}(.),{\theta}_{2}(.),{w}_{2}(.)}(\Omega)$ are bounded.
\end{theorem}

In~{\cite{Guliyev}} it was proved that, if $\alpha(x)$ is a variable, under the conditions of Theorem~{\ref{th2.2}} the operators
$\frac{1}{{(1+|x|)}^{\gamma(x)}}{M}^{\alpha(.)}$ and $\frac{1}{{(1+|x|)}^{\gamma(x)}}{I}^{\alpha(.)}$ from ${M}_{p(.),{w}_{1}(.)}$ to ${M}_{q(.),{w}_{2}(.)}$ are bounded, where $\gamma(x)={A}_{\infty}\alpha(x)[1-\frac{\alpha(x)}{n}]\leq\frac{n}{4}{A}_{\infty}$.

The next lemma was proved in~{\cite{Guliyev}}.

\begin{lemma}\label{lem2.1}
Assume that $p\in{\mathbb{P}}_{\infty}^{log}(\Omega)$ and $f\in{L}_{loc}^{p(.)}({\mathbb R}^{n})$.Then
$$
{\|f\|}_{{L}_{p(.)}(B(x,t))}\leq{C}{t}^{{\eta}_{p}(x,t)}\int_{t}^{\infty}{r}^{-{\eta}_{p}(x,r)-1}{\|f\|}_{{L}_{p(.)}(B(x,r))}dr.
$$
\end{lemma}

In the same paper, the next theorem was proved.
\begin{theorem}\label{th2.4} Suppose that  $p\in{\mathbb{P}}_{\infty}^{log}(\Omega)$.Then
$$
{||Mf||}_{{L}_{p(.)}(\tilde{B}(x,t))}\leq{C}{t}^{{\eta}_{p}(x,t)}\mathop{sup}_{r>2t}{r}^{-{\eta}_{p}(x,r)}{||f||}_{{L}_{p(.)}(\tilde{B}(x,r))},
$$
for every $f\in{L}_{p(.)}(\Omega)$, where $C$ is independent of $f,x\in\Omega$ amd $t>0$.
\end{theorem}
We prove the next necessary inequality.
\begin{theorem} Let $p\in{\mathbb{P}}_{\infty}^{log}(\Omega)$. Then
\begin{equation}\label{eq2.7}
{||Mf||}_{{L}_{p(.)}(\tilde{B}(x,t))}\leq{C}{t}^{{\eta}_{p}(x,t)}\int_{t}^{\infty}{s}^{-{\eta}_{p}(x,s)-1}{||f||}_{{L}_{p(.)}(\tilde{B}(x,s))}ds,
\end{equation}
where $C$ is independent of $f,x,t$.
\end{theorem}
\begin{proof}
Using Theorem~{\ref{th2.4}} and Lemma~{\ref{lem2.1}}, we have
$$
{||Mf||}_{{L}_{p(.)}(\tilde{B}(x,t))}\leq{C}{t}^{{\eta}_{p}(x,t)}\mathop{sup}_{r>2t}{r}^{-{\eta}_{p}(x,r)}{||f||}_{{L}_{p(.)}(\tilde{B}(x,r))}\leq{C}{t}^{{\eta}_{p}(x,t)}\mathop{sup}_{r>t}{r}^{-{\eta}_{p}(x,r)}{||f||}_{{L}_{p(.)}(\tilde{B}(x,r))}
$$
$$
\leq{C}{t}^{{\eta}_{p}(x,t)}\mathop{sup}_{r>t}\int_{r}^{\infty}{s}^{-{\eta}_{p}(x,s)-1}{\|f\|}_{{L}_{p(.)}(B(x,s))}ds={C}{t}^{{\eta}_{p}(x,t)}\int_{t}^{\infty}{s}^{-{\eta}_{p}(x,s)-1}{||f||}_{{L}_{p(.)}(\tilde{B}(x,s))}ds.
$$
\end{proof}
The next inequality was proved in~{\cite{Guliyev}}.
\begin{theorem}\label{th2.6} Let $p\in{\mathbb{P}}_{\infty}^{log}(\Omega)$ and $\alpha$, $q$ satisfy conditions $0<\alpha<n$, $\frac{1}{q(x)}=\frac{1}{p(x)}-\frac{\alpha}{n}$.
Then the next inequality holds
\begin{equation}\label{eq2.8}
{||{I}^{\alpha}f||}_{{L}_{q(.)}(\tilde{B}(x,t))}\leq{C}{t}^{{\eta}_{q}(x,t)}\int_{t}^{\infty}{r}^{-{\eta}_{q}(x,r)-1}{||f||}_{{L}_{p(.)}(\tilde{B}(x,r))}dr,
t>0
\end{equation}
where $C$ is independent of $x$ and $t$.
\end{theorem}
The inequality~{\eqref{eq2.8}} holds, if we put $\frac{1}{{(1+|y|)}^{\gamma(y)}}{I}^{\alpha(.)}f(y)$ instead of ${I}^{\alpha}f(x)$. Namely, the following is true.
\begin{theorem} Let $p\in{\mathbb{P}}_{\infty}^{log}(\Omega)$ and the function  $\alpha(x)$, $q(x)$ satisfy the condition $\frac{1}{q(x)}=\frac{1}{p(x)}-\frac{\alpha(x)}{n}$.
Then for each ${x}\in{\mathbb{R}}^{n}, t>0$ the following inequality holds
\begin{equation}\label{eq2.9}
{||\frac{1}{{(1+|y|)}^{\gamma(y)}}{I}^{\alpha(.)}f||}_{{L}_{q(.)}(\tilde{B}(x,t))}\leq{C}{t}^{{\eta}_{q}(x,t)}{\int}_{t}^{\infty}{r}^{-{\eta}_{q}(x,r)-1}{||f||}_{{L}_{p(.)}(\tilde{B}(x,r))}dr
\end{equation}
\end{theorem}
\begin{proof}
We represent the function $f$ as
$f(x)={f}_{1}(x)+{f}_{2}(x)$,
${f}_{1}(x)=f(x){\chi}_{\tilde{B}(x,2t)}$,
${f}_{2}(x)=f(x){\chi}_{{\Omega}\backslash{\tilde{B}(x,2t)}}$.
Then
$$
\frac{1}{{(1+|y|)}^{\gamma(y)}}{I}^{\alpha(.)}f(y)=\frac{1}{{(1+|y|)}^{\gamma(y)}}{I}^{\alpha(.)}{f}_{1}(y)+\frac{1}{{(1+|y|)}^{\gamma(y)}}{I}^{\alpha(.)}{f}_{2}(y).
$$
By Theorem~{\ref{th2.3}},
$$
{||\frac{1}{{(1+|y|)}^{\gamma(y)}}{I}^{\alpha(.)}{f}_{1}||}_{{L}_{q(.)}(\tilde{B}(x,t))}\leq{||\frac{1}{{(1+|y|)}^{\gamma(y)}}{I}^{\alpha(.)}{f}_{1}||}_{{L}_{q(.)}(({R}^{n})}\leq{C}{\|{f}_{1}\|}_{{L}_{p(.)}({R}^{n})}=C{\|f\|}_{{L}_{p(.)}(B(x,2t))}.
$$
By Lemma~{\ref{lem2.1}},
\begin{equation}\label{eq2.10}
{||\frac{1}{{(1+|y|)}^{\gamma(y)}}{I}^{\alpha(.)}{f}_{1}||}_{{L}_{q(.)}(\tilde{B}(x,t))}\leq{C}{t}^{{\eta}_{q}(x,t)}\int_{2t}^{\infty}{r}^{-{\eta}_{q}(x,r)-1}{||f||}_{{L}_{p(.)}(\tilde{B}(x,r))}dr.
\end{equation}
If $|x-z|\leq{t}$ and $|z-y|\geq2t$, we have $\frac{1}{2}|z-y|\leq|x-y|\leq\frac{3}{2}|z-y|$. Using the inequality $\frac{1}{{(1+|y|)}^{\gamma(y)}}\leq1$, we infer
$$
{||\frac{1}{{(1+|y|)}^{\gamma(y)}}{I}^{\alpha(.)}{f}_{2}||}_{{L}_{q(.)}(\tilde{B}(x,t))}\leq{\|\int_{{R}^{n}\backslash{B(x,2t)}}{|z-y|}^{\alpha-n}f(y)dy\|}_{{L}_{q(.)}(\tilde{B}(x,t))}
$$
$$
\leq{C}\int_{{R}^{n}\backslash{B(x,2t)}}{|x-y|}^{\alpha-n}|f(y)|dy{\|{\chi}_{B(x,t)}\|}_{{L}_{q(.)}({R}^{n})}.
$$
Choosing $\beta>\frac{n}{{q}_{-}}$, we obtain
$$
\int_{{R}^{n}\backslash{B(x,2t)}}{|x-y|}^{\alpha-n}|f(y)|dy=\beta\int_{{R}^{n}\backslash{B(x,2t)}}{|x-y|}^{\alpha-n+\beta}|f(y)|(\int_{|x-y|}^{\infty}{s}^{-\beta-1}ds)dy
$$
=
$$
=\beta\int_{2t}^{\infty}{s}^{-\beta-1}(\int_{y\in{R}^{n}:2t\leq|x-y|\leq{s}}{|x-y|}^{\alpha-n+\beta}|f(y)|dy)ds
$$
$$
\leq{C}\int_{2t}^{\infty}{s}^{-\beta-1}{\|f\|}_{{L}_{p(.)}(B(x,s))}{\|{|x-y|}^{\alpha-n+\beta}\|}_{{L}_{{p}^{'}(.)}(B(x,s))}ds
$$
$$
\leq{C}\int_{2t}^{\infty}{s}^{\alpha-{\eta}_{p}(x,s)-1}{\|f\|}_{{L}_{p(.)}(B(x,s))}ds.
$$
Therefore
$$
{||\frac{1}{{(1+|y|)}^{\gamma(y)}}{I}^{\alpha(.)}{f}_{2}||}_{{L}_{q(.)}(\tilde{B}(x,t))}\leq{C}{t}^{{\eta}_{p}(x,t)}\int_{2t}^{\infty}{s}^{-{\eta}_{q}(x,s)-1}{\|f\|}_{{L}_{p(.)}(B(x,s))}ds
$$
which, together with~{\eqref{eq2.10}}, yields~{\eqref{eq2.9}}.
\end{proof}
Let $u$ and $v$ be a positive measurable functions.The dual Hardy operator is defined by the identity
\begin{equation*}
{\tilde{H}}_{v,u}f(x)=v(x)\int_{x}^{\infty}f(t)u(t)dt,  \ x\in{\mathbb R}^{n}.
\end{equation*}
Suppose that $a$ is a positive fixed number. Let ${\theta}_{1,a}(x)=\mathop{essinf}_{y\in[x,a)}{\theta}_{1}(y),$
$$
{\tilde{\theta}}_{1}(x)=\begin{cases}
{\theta}_{1,a}(x)  &\text{if $x\in[0,a]$;}\\
{\overline{\theta}}_{1}=const  &\text{if $x\in[a,\infty)$;}\\
\end{cases},
{\mathbb{\theta}}_{1}=\mathop{essinf}_{x\in{R}_{+}}{\theta}_{1}(x),
{\Theta}_{2}=\mathop{essup}_{x\in{R}_{+}}{\theta}_{2}(x)
$$
The next theorem was proved in~{\cite{EKM}}.
\begin{theorem}\label{th2.8} Let ${\theta}_{1}(x)$ and  ${\theta}_{2}(x)$ measurable functions on ${R}_{+}$. Suppose that there exists a positive number $a$  for all $x>a$ holds  ${\theta}_{1}(x)=\overline{{\theta}}_{1}=const$, ${\theta}_{2}(x)=\overline{{\theta}_{2}}=const$ and  $1<{\theta}_{1}\leq{\tilde{\theta}}_{1}(x)\leq{\theta}_{2}(x)\leq{{\Theta}_{2}}<\infty$ for a.a.
If
\begin{equation*}
G=\mathop{sup}_{t>0}\int_{0}^{t}{[v(x)]}^{{\theta}_{2}(x)}{(\int_{t}^{\infty}{u}^{{\tilde{\theta}_{1}}^{'}(x)}(\tau)d\tau)}^{\frac{{\theta}_{2}(x)}{{({\theta}_{1})}^{'}(x)}}dx<\infty
\end{equation*}
then the operator ${\tilde{H}}_{v,u}$ is bounded from ${L}_{{\theta}_{1}(.)}({R}^{+})$ to ${L}_{{\theta}_{2}(.)}({R}^{+})$.
\end{theorem}

\section{The Main Results}
\begin{theorem}\label{th3.1} Assume that $p(.)\in{\mathbb{P}}_{\infty}^{log}(\Omega)$, and ${\theta}_{1}(x)$ and  ${\theta}_{2}(x)$ are measurable functions on ${R}_{+}$. Suppose that there exists a positive number $a$ such that for all $t>a$ we have ${\theta}_{1}(x)=\overline{{\theta}}_{1}=const$, ${\theta}_{2}(x)=\overline{{\theta}_{2}}=const$ and  $1<{\theta}_{1}\leq{\tilde{\theta}}_{1}(x)\leq{\theta}_{2}(x)\leq{{\Theta}_{2}}<\infty$ for a.a., the positive measurable functions ${w}_{1}$ и ${w}_{2}$ satisfy the condition
\begin{equation}\label{eq3.1}
A=\mathop{sup}_{x\in\Omega,t>0}\int_{0}^{t}{({w}_{2}(x,r))}^{{\theta}_{2}(r)}{\Bigl(\int_{t}^{\infty}{\Bigl(\frac{1}{{w}_{1}(x,s)s}\Bigr)}^{{[\tilde{\theta}_{1}(r)]}^{'}}ds\Bigr)}^{\frac{{\theta}_{2}(r)}{{[\tilde{\theta}_{1}(r)]}^{'}}}dr<\infty
\end{equation}
Then the maximal operator $M$ from ${GM}_{p(.),{\theta}_{1}(.),{w}_{1}(.)}(\Omega)$ to ${GM}_{p(.),{\theta}_{2}(.),{w}_{2}(.)}(\Omega)$ is bounded.
\end{theorem}
\begin{corollary}\label{cor3.1}Let $p(.)\in{\mathbb{P}}_{\infty}^{log}(\Omega)$, ${w}_{1}(x,r)={w}_{2}(x,r)={r}^{\beta(x)}$. If
\begin{equation}\label{eq3.2}
\mathop{inf}_{x\in\Omega,r>0}(\beta(x)+1){[\tilde{\theta}_{1}(r)]}^{'}>1,
\end{equation}
\begin{equation}\label{eq3.3}
\mathop{sup}_{x\in\Omega,t>0}\int_{0}^{t}{r}^{{\theta}_{2}(r)\beta(x)}\frac{{t}^{[-(\beta(x)+1){[\tilde{\theta}_{1}(r)]}^{'}+1]{\frac{{\theta}_{2}(r)}{{[\tilde{\theta}_{1}(r)]}^{'}}}}}{{[(\beta(x)+1){[\tilde{\theta}_{1}(r)]-1}]}^{\frac{{\theta}_{2}(r)}{{[\tilde{\theta}_{1}(r)]}^{'}}}}dr<\infty.
\end{equation}
Then the maximal operator $M$ from ${GM}_{p(.),{\theta}_{1}(.),{r}^{\beta(.)}}(\Omega)$ to ${GM}_{p(.),{\theta}_{2}(.),{r}^{\beta(.)}}(\Omega)$ is bounded.
\end{corollary}
The following theorems give Spanne-type results on the boundedness of the Riesz potential ${I}^{\alpha}$ in global Morrey-type spaces with variable exponent ${GM}_{p(.),\theta(.),w(.)}(\Omega)$.
In the following theorem $\alpha=const$.
\begin{theorem}\label{th3.2}
Assume that $p(.)\in{\mathbb{P}}_{\infty}^{log}(\Omega)$, the constant number  $\alpha$ satisfies the condition  $\alpha>0$, ${(\alpha p(.))}_{+}=\mathop{sup}_{x\in\Omega}\alpha p(x)<n$. Let ${\theta}_{1}(x)$ and  ${\theta}_{2}(x)$ be measurable functions on ${R}_{+}$. Suppose that there exists a positive number $a$ such that for all $x>a$ we have ${\theta}_{1}(x)=\overline{{\theta}}_{1}=const$, ${\theta}_{2}(x)=\overline{{\theta}_{2}}=const$ and  $1<{\theta}_{1}\leq{\tilde{\theta}}_{1}(x)\leq{\theta}_{2}(x)\leq{{\Theta}_{2}}<\infty$ for a.a., the functions  ${p}_{1}(x)$ и ${p}_{2}(x)$ satisfy  $\frac{1}{{p}_{2}(x)}=\frac{1}{{p}_{1}(x)}-\frac{\alpha(x)}{n}$, and the functions ${w}_{1}$ и ${w}_{2}$ satisfy the condition
\begin{equation}\label{eq3.4}
T=\mathop{sup}_{x\in\Omega,t>0}\int_{0}^{t}{({w}_{2}(x,r))}^{{\theta}_{2}(r)}{\Bigl(\int_{t}^{\infty}{\Bigl(\frac{{s}^{\alpha-1}}{{w}_{1}(x,s)}\Bigr)}^{{[\tilde{\theta}_{1}(r)]}^{'}}ds\Bigr)}^{\frac{{\theta}_{2}(r)}{{[\tilde{\theta}_{1}(r)]}^{'}}}dr<\infty.
\end{equation}
Then the operators ${I}_{\alpha}$ и ${M}_{\alpha}$ from ${GM}_{{p}_{1}(.),{\theta}_{1}(.),{w}_{1}(.))}(\Omega)$ to ${GM}_{{p}_{2}(.),{\theta}_{2}(.),{w}_{2}(.)}(\Omega)$ are bounded.
\end{theorem}
\begin{corollary}\label{cor3.2}Let $p(.)\in{P}_{\infty}^{log}(\Omega)$, ${w}_{1}(x,r)={w}_{2}(x,r)={r}^{\beta(x)}$. If
\begin{equation}\label{eq3.5}
\mathop{sup}_{x\in\Omega,r>0}(\alpha-\beta(x)-1){[\tilde{\theta}_{1}(r)]}^{'}<-1,
\end{equation}
\begin{equation}\label{eq3.6}
\mathop{sup}_{x\in\Omega,t>0}\int_{0}^{t}{r}^{{\theta}_{2}(r)\beta(x)}\frac{{t}^{[[\alpha-\beta(x)-1]{[\tilde{\theta}_{1}(r)]}^{'}+1]\frac{{\theta}_{2}(r)}{{[\tilde{\theta}_{1}(r)]}^{'}}}}{[\beta(x)+1-\alpha]{[\tilde{\theta}_{1}(r)]}^{'}-1}dr<\infty,
\end{equation}
then the operators ${I}_{\alpha}$ и ${M}_{\alpha}$ from ${GM}_{p(.),{\theta}_{1}(.),{r}^{\beta}}(\Omega)$ to ${GM}_{p(.),{\theta}_{2}(.),{r}^{\beta}}(\Omega)$ are bounded.
\end{corollary}
In the following theorem $\alpha(x)$ is a variable exponent.
\begin{theorem}\label{th3.3} Assume that $p(.)\in{\mathbb{P}}_{\infty}^{log}(\Omega)$, the function $\alpha(x)$ satisfies the condition $\alpha(x)>0$, ${(\alpha(.) p(.))}_{+}=\mathop{sup}_{x\in\Omega}\alpha(x) p(x)<n$. Let ${\theta}_{1}(x)$ and  ${\theta}_{2}(x)$ be measurable functions on ${R}_{+}$. Suppose that there exists a positive number $a$ such that for all $x>a$ we have ${\theta}_{1}(x)=\overline{{\theta}}_{1}=const$, ${\theta}_{2}(x)=\overline{{\theta}_{2}}=const$ and  $1<{\theta}_{1}\leq{\tilde{\theta}}_{1}(x)\leq{\theta}_{2}(x)\leq{{\Theta}_{2}}<\infty$ for a.a., the functions ${p}_{1}(x)$ и ${p}_{2}(x)$ satisfy  $\frac{1}{{p}_{2}(x)}=\frac{1}{{p}_{1}(x)}-\frac{\alpha(x)}{n}$, the functions ${w}_{1}$ и ${w}_{2}$ satisfy the condition
\begin{equation*}
T=\mathop{sup}_{x\in\Omega,t>0}\int_{0}^{t}{({w}_{2}(x,r))}^{{\theta}_{2}(r)}{\Bigl(\int_{t}^{\infty}{\Bigl(\frac{{s}^{\alpha(x)-1}}{{w}_{1}(x,s)}\Bigr)}^{{[\tilde{\theta}_{1}(r)]}^{'}}ds\Bigr)}^{\frac{{\theta}_{2}(r)}{{[\tilde{\theta}_{1}(r)]}^{'}}}dr<\infty.
\end{equation*}
Then the operators $\frac{1}{{(1+|x|)}^{\gamma(x)}}{I}^{\alpha(.)}$ и $\frac{1}{{(1+|x|)}^{\gamma(x)}}{M}^{\alpha(.)}$ from ${GM}_{{p}_{1}(.),{\theta}_{1}(.),{w}_{1}(.))}(\Omega)$ to ${GM}_{{p}_{2}(.),{\theta}_{2}(.),{w}_{2}(.)}(\Omega)$ are bounded.
\end{theorem}
In the following theorem $\alpha=const$.
\begin{theorem}
Assume that $p(.)\in{P}_{\infty}^{log}(\Omega)$, bounded measurable functions ${\theta}_{1}(x)$,
${\theta}_{2}(x)$ satisfy $1<{\theta}_{1-}(.)\leq{\theta}_{1}(.)\leq{\theta}_{1+}(.)$, $1<{\theta}_{2-}(.)\leq{\theta}_{2}(.)\leq{\theta}_{2+}(.)$, and there exists $a$ such that for all $t>a$ we have ${\theta}_{1}(t)=const$, ${\theta}_{2}(t)=const$. Let the positive measurable functions ${w}_{1}(x,r)$, ${w}_{2}(x,r)$ satisfy the condition
$$
\mathop{sup}_{x\in\Omega}{\Bigl\|{w}_{2}(x,r){||\frac{1}{t{w}_{1}(x,t)}||}_{{L}_{{\theta}_{1}^{'}(.)}(r,\infty)}\Bigr\|}_{{L}_{{\theta}_{2}(.)}(0,\infty)}<\infty.
$$
Then the singular integral operator $T$ from ${GM}_{p(.),{\theta}_{1}(.),{w}_{1}(.)}$ to ${GM}_{p(.),{\theta}_{2}(.),{w}_{2}(.)}$ is bounded.
\end{theorem}
\begin{proof}[Proof of Theorem~${\ref{th3.1}}$] According to the definition and to Theorem~{\ref{th2.8}}, using the Holder inequality with variable exponents $\theta$, ${\theta}^{'}$, we infer
$$
{||Mf||}_{{GM}_{p(.),{\theta}_{2}(.),{w}_{2}(.)}(\Omega)}=\mathop{sup}_{x\in\Omega}{||\frac{{w}_{2}(x,r)}{{r}^{{\eta}_{p}(x,r)}}{||Mf||}_{{L}_{p(.)}(B(x,r))}||}_{{L}_{{\theta}_{2}(.)}(0,\infty)}
$$
$$
\leq{C}\mathop{sup}_{x\in\Omega}{||{{w}_{2}(x,r)}\int_{r}^{\infty}{t}^{-{\eta}_{p}(x,t)-1}{||f||}_{{L}_{p(.)}(B(x,t))}dt||}_{{L}_{{\theta}_{2}(.)}(0,\infty)},
$$
we denote
$$
{\tilde{H}}_{v,u}f(r)=v(r)\int_{r}^{\infty}g(t)u(t)dt,
$$
here
$v(r)={w}_{2}(x,r)$,
$g(t)=\frac{{w}_{1}(x,t)}{{t}^{{\eta}_{p}(x,t)}}{||f||}_{{L}_{p(.)}(B(x,t))}$,
$u(t)=\frac{1}{{w}_{1}(x,t)t}$ for every fixed $x\in\Omega$.
Then the condition~{\eqref{eq2.7}} has the form~{\eqref{eq3.1}}, from which boundedness of the operator ${\tilde{H}}_{v,w}f(r)$  from ${L}_{{\theta}_{1}(.)}(0,\infty)$ to ${L}_{{\theta}_{2}(.)}(0,\infty)$ follows. Consequently, we have
$$
{\|Mf\|}_{{GM}_{q(.),{\theta}_{2}(.),{w}_{2}(.)}(\Omega)}\leq{A}\cdot\mathop{sup}_{x\in\Omega}{\|{w}_{1}(x,t){t}^{-{\eta}_{p}(x,t)}{\|f\|}_{{L}_{p(.)}(B(x,t))}\|}_{{L}_{{\theta}_{1}(.)}(0,\infty)}=
$$
$$
=A\cdot{\|f\|}_{{GM}_{p(.),{\theta}_{1}(.),{w}_{1}(.)}},
$$
this means that the operator $M$ from ${GM}_{p(.),{\theta}_{1}(.),{w}_{1}(.)}$ to ${GM}_{q(.),{\theta}_{2}(.),{w}_{2}(.)}$ is bounded.
\end{proof}
\begin{proof}[Proof of Corollary~${\ref{cor3.1}}$]
The condition~{\eqref{eq3.1}} has the form
$$
\mathop{sup}_{x\in\Omega,t>0}\int_{0}^{t}{r}^{{\theta}_{2}(r)\beta(x)}{\Bigl(\int_{t}^{\infty}{s}^{-[\beta(x)+1]{[\tilde{\theta}_{1}(r)]}^{'}}ds\Bigr)}^{\frac{{\theta}_{2}(r)}{{[\tilde{\theta}_{1}(r)]}^{'}}}dr<\infty.
$$
By the convergence of the inner integral, we obtain the conditions~{\eqref{eq3.2}} and~{\eqref{eq3.3}}.
\end{proof}
\begin{proof}[Proof of Theorem~${\ref{th3.2}}$]
Using the definition and the Theorem~{\ref{th2.6}}, we have
$$
{||{I}^{\alpha}||}_{{GM}_{q(.),{\theta}_{2}(.),{w}_{2}(.)}(\Omega)}=\mathop{sup}_{x\in\Omega}{\|{w}_{2}(x,r){r}^{-{\eta}_{q}(x,r)}{\|{I}_{\alpha}f\|}_{{L}_{q(.)}(B(x,r))}\|}_{{L}_{{\theta}_{2}(.)}(0,\infty)}\leq
$$
$$
\leq{C}\mathop{sup}_{x\in\Omega}{||{w}_{2}(x,r)\int_{r}^{\infty}{t}^{-{\eta}_{q}(x,t)-1}{||f||}_{{L}_{p(.)}(B(x,t))}dt||}_{{L}_{{\theta}_{2}(.)}(0,\infty)}
$$
We denote
$$
{\tilde{H}}_{v,u}f(r)=v(r)\int_{r}^{\infty}g(t)u(t)dt,
$$
we denote by
$v(r)={w}_{2}(x,r)$,
$g(t)=\frac{{w}_{1}(x,t)}{{t}^{{\eta}_{p}(x,t)}}{||f||}_{{L}_{p(.)}(B(x,t))}$,
$u(t)=\frac{{t}^{{\eta}_{p}(x,t)-{\eta}_{q}(x,t)-1}}{{w}_{1}(x,t)}$ for every fixed $x\in\Omega$.
Then the condition~{\eqref{eq2.7}} has the form~{\eqref{eq3.4}}, from which boundedness of the operator ${\tilde{H}}_{v,w}f(r)$  from ${L}_{{\theta}_{1}(.)}(0,\infty)$ to ${L}_{{\theta}_{2}(.)}(0,\infty)$ follows. Consequently, we have
$$
{\|{I}^{\alpha}f\|}_{{GM}_{q(.),{\theta}_{2}(.),{w}_{2}(.)}(\Omega)}\leq{T}\cdot\mathop{sup}_{x\in\Omega}{\|{w}_{1}(x,t){t}^{-{\eta}_{p}(x,t)}{\|f\|}_{{L}_{p(.)}(B(x,t))}\|}_{{L}_{{\theta}_{1}(.)}(0,\infty)}=
$$
$$
=T\cdot{\|f\|}_{{GM}_{p(.),{\theta}_{1}(.),{w}_{1}(.)}},
$$
this means that the operator ${I}^{\alpha}$ from ${GM}_{{p}_{1}(.),{\theta}_{1}(.),{w}_{1}(.)}$ to ${GM}_{{p}_{2}(.),{\theta}_{2}(.),{w}_{2}(.)}$ is bounded.
\end{proof}
\begin{proof}[Proof of Corollary~${\ref{cor3.2}}$]
The condition~{\eqref{eq3.4}} takes the form
$$
\mathop{sup}_{x\in\Omega,t>0}\int_{0}^{t}{r}^{{\theta}_{2}(r)\beta(x)}{\Bigl(\int_{t}^{\infty}{s}^{(\alpha-1-\beta(x)){[\tilde{\theta}_{1}(r)]}^{'}}ds\Bigr)}^{{\frac{{\theta}_{2}(r)}{{[\tilde{\theta}_{1}(r)]}^{'}}}}dr<\infty.
$$
By the convergence of the inner integral, we deduce the conditions~{\eqref{eq3.5}} and~{\eqref{eq3.6}}.
\end{proof}
\begin{proof}[Proof of Theorem~${\ref{th3.3}}$] The proof of this theorem the same as the proof of Theorem~{\ref{th3.2}}, it is sufficient to put $\frac{1}{{(1+|x|)}^{\gamma(x)}}{I}^{\alpha(.)}f(x)$ instead of ${I}^{\alpha}f(x)$.
\end{proof}





\def\bibname{\vspace*{-30mm}{

\vskip 1 cm \footnotesize
\begin{flushleft}
Nurzhan~Bokayev, Zhomart~Onerbek \\ 
Department of Mechanics and Mathematics\\ 
L.N. Gumilyov Eurasian National University \\ 
5 Munaitpasov St,\\ 
010008 Astana, Nur-Sultan, Kazakhstan \\ 
E-mail: onerbek.93@mail.ru
\end{flushleft}


 \end{document}